\documentclass[11pt, a4paper]{article}

\usepackage[utf8]{inputenc}
\usepackage{amsmath, amssymb, amsthm}
\usepackage[english]{babel}

\usepackage[margin=1in]{geometry}
\usepackage{lmodern}
\usepackage{graphicx}
\usepackage{authblk}
\usepackage{hyperref}
\usepackage[dvipsnames]{xcolor}
\usepackage{nicematrix}
\hypersetup{
    colorlinks=true,
    linkcolor=blue,
    filecolor=magenta,      
    urlcolor=cyan,
    pdftitle={Asymptotic Characters},
    pdfpagemode=FullScreen,
}

\DeclareMathOperator{\sgn}{sgn}
\DeclareMathOperator{\Stab}{Stab}
\newcommand{\dimL}{\mathrm{dim}\,L_\lambda}
\newcommand{\su}{\mathfrak{su}}

\newtheorem{theorem}{Theorem}[section]
\newtheorem{lemma}[theorem]{Lemma}

\newtheorem{corollary}[theorem]{Corollary} 
\theoremstyle{definition}

\theoremstyle{remark}

\title{The high-dimension limit of characters of compact reductive Lie groups and restrictions on the production of quantum randomness}

\author[1]{Piotr Borodako kf}
\author[2]{Adam Sawicki}
\affil[1]{The College of Interfaculty Individual Studies in Mathematics and Natural Science, University of Warsaw}
\affil[2]{Center for Theoretical Physics, Polish Academy of Sciences,
Al. Lotników 32/46, 02-668 Warszawa, Poland}
\date{\today}

\begin{document}

\maketitle

\begin{abstract}
\noindent
For any element $g$ of compact reductive group $G$, we investigate the asymptotic behavior of its normalized irreducible character in the high-dimension limit, $\frac{\chi_\lambda(g)}{d_\lambda}$. We show that when $G$ is simple the limit vanishes besides identity element. For semisimple groups, one gets the same results under the  additional assumption that dimensions of irreducible representations of all simple components are going to infinity.

We apply this result to the study of pseudorandom quantum evolution. This lets us establish a bound on the generation of the approximate t-designs that do not depend on the underlying symmetry group but only on the number of the generators used. The limiting spectral distribution turns out to be governed by the Kesten-McKay law.


\end{abstract}

\section{Introduction and motivations}
\label{sec:intro}

Approximate $t$-designs are collections of unitary operators that mimic the Haar measure by reproducing, up to order $t$, the averages of polynomial functions in the matrix elements, though only approximately. A connection between such designs and $\varepsilon$-nets was recently clarified in \cite{nets}. Although constructions of exact $t$-designs are known \cite{Yoshifumi2}, implementing them on current quantum devices is hindered by unavoidable noise and control imperfections, which in practice turns them into approximate $t$-designs. These structures are widely used across quantum information science: examples include randomized benchmarking \cite{Gambetta2014}, efficient learning and estimation of quantum states \cite{EffLearning2020}, decoupling protocols \cite{Decoupl2013}, quantum communication tasks \cite{InfTransmission2009}, state discrimination problems \cite{StateDiscrimination2005}, characterizing universal gate sets \cite{Sawicki22}, and studying the growth of complexity in quantum systems \cite{Suskind2018,ChaosDesign2017,MO1,jonas2}. 

The question of how to efficiently realize pseudo-random unitaries has recently attracted significant attention. Early progress in \cite{HArrowHasstings2008,HH09} demonstrated that Haar-random subsets of $U(d)$ already form good $t$-designs for $t = O(d^{1/6}/\log d)$. Later, it was shown in \cite{BHH2016} that random quantum circuits consisting of Haar-random two-qubit gates, arranged with a prescribed architecture and polynomial depth in the number of qubits $N$, generate approximate $t$-designs. Subsequent advances \cite{Harrow2018,Haferkamp2022randomquantum,PhysRevA.104.022417,Yoshifumi1} established even faster convergence with increasing system size. In another direction, \cite{Qhomeopathy2020} proved that circuits built mainly from Clifford gates, supplemented by a small number of non-Clifford gates, can efficiently approximate $t$-designs. Research concerning global bounds remains highly active; recently, progress was made in establishing simple, linear depth circuit construction \cite{Metger2024} and highly efficient designs generated from random sums and permutations \cite{Chen2024}. Nonetheless, numerical and theoretical evidence in \cite{OLIVIERO2021127721,Leone2021quantumchaosis} suggests that achieving full $t$-design behavior in realistic settings often demands a substantial number of non-Clifford resources.


When designing a practical implementation, one is inevitably led to approximate t-designs, raising the critical question of how to quantify their quality. Directly verifying the design properties of a given ensemble is known to be computationally demanding \cite{Nakata2025}. The answer to this question requires investigating the norm of so-called moment operators. Useful lower bounds on such norms can be obtained by studying normalized characters in high-dimensional representations, i.e., the asymptotic behavior of the quantity 
\begin{equation}
    \frac{\chi_\lambda(g)}{\mathrm{dim}V_\lambda}.
\end{equation}
The purpose of this work is to rigorously analyze this limit for any reductive compact Lie group. The main result of our paper is Theorem \ref{cor:critdiv}.

The article is structured as follows. In Section \ref{sec:wcf}, we briefly outline the theoretical tools employed throughout our proof. In Section \ref{sec:su2_example}, we provide an illustrative calculation for the SU$(2)$ case, where the validity of our thesis can be immediately seen. This elementary example serves as the primary motivation for our hypothesis that the result holds in general.

The core of our analysis is presented in Section \ref{sec:analysis}, which is divided into two parts. The first part Section \ref{sec:su3_case_study} is a detailed discussion of the SU$(3)$ case; it is a rich special case that contains most of the core difficulties of the general problem, yet all of its objects can be written out directly. The second part of Section \ref{sec:sun_proof} concerns the general SU$(N)$ case and focuses on the formal generalization of the insights gained from SU$(3)$. There, we prove the phenomena observed explicitly in the specific case and provide the theoretical reasons why they hold in general. In Section \ref{sec:discussion}, we discuss the interpretation of the proof. We investigate the geometric meaning of the formulas used and the structure of the mathematical objects that emerge in the limit. Then, in Section \ref{sec:generalization}, we present a generalization of the result for all compact reductive Lie groups. Finally, in Section \ref{sec:quantumpart}, we discuss the implications of our findings for pseudo-random unitaries.


We note that results similar to ours were published in the past in \cite{SizeCharactersCompact1991}. However, the proof there is purely analytical, based on constructing an upper-bound, contrary to our direct and constructive algebraic approach. Thanks to it we can understand that obtained formulas have a clear algebraic and geometric meaning that turns out to relate to the centralizer of the considered element. Thanks to this insight, one is able to extend the reasoning conducted in Section \ref{sec:analysis} to all reductive compact algebras in Section \ref{sec:generalization} leading to a sharp criterion in Theorem \ref{cor:critdiv} that precisely delineates when the character vanishes. 

The crux of our proof is Lemma \ref{lemma:divergence} establishing the divergence of a product over non-degenerate roots. Our proof of this lemma is constructive, providing a method based on the connectivity of the Dynkin diagram that shows exactly how to find the root responsible for the divergence. This explicitness might be of use for further research on the topic.

In order to give an interpretation to our results one can think of a quantum state moving under the symmetry of a compact group. Each step of the random walk applies one of the generators from a fixed set, and averaging over these steps describes how randomness spreads across the Hilbert space. The spectral gap of this averaging operator tells us how quickly the process “forgets” its initial state and approaches true randomness.

Our result shows a kind of universality: no matter which compact symmetry group is chosen, once the representation is large the spectral behavior always reduces to the same distribution, known from random walks on infinite trees (the Kesten–McKay law). This universality has a practical message for quantum information: the efficiency of generating  unitary t-designs cannot exceed a fixed bound. The only thing that matters is the number of random generators one uses. In short, symmetry cannot accelerate the production of quantum randomness beyond this universal speed limit.

\section{Theoretical Framework: The Weyl Character Formula}
\label{sec:wcf}

When studying the representation theory of Lie algebras, we often limit our attention to non-abelian algebras which do not contain any non-trivial ideals—we call them simple, and algebras that are a direct sum of them, which we call semi-simple. The theory of representations of these algebras is particularly well-behaved due to Weyl's Theorem on Complete Reducibility \cite{kirillovIntroductionLieGroups2017}, which states that every finite-dimensional representation of a semi-simple Lie algebra is a direct sum of irreducible representations.

In the following, we will assume that $\mathfrak{g}$ is a simple Lie algebra. The key to understanding $\mathfrak{g}$ is looking at its maximal abelian subalgebra, called the \textbf{Cartan Subalgebra} (CSA), $\mathfrak{h}$. It is a foundational result that every element of $\mathfrak{g}$ is conjugate to an element of $\mathfrak{h}$ under the Adjoint action of the corresponding Lie group.

We can now perceive the whole algebra $\mathfrak{g}$ as a linear space on which our chosen CSA acts via the adjoint representation,
\begin{equation}
    \mathrm{ad}: \mathfrak{g} \to \mathfrak{gl}(\mathfrak{g}), \quad \mathrm{ad}(X)(Y) = [X,Y].
\end{equation}

As the acting algebra is abelian, its representation consists of a set of commuting matrices. For a compact real form like $\su(N)$, these matrices can be simultaneously diagonalized, and thus they decompose our Lie algebra into invariant subspaces on which $\mathfrak{h}$ acts by scalar multiplication. The number by which we multiply is linearly dependent on the element of $\mathfrak{h}$ and is thus a functional on the Cartan subalgebra. We call these non-zero functionals \textbf{roots} ($\alpha$). The one-dimensional subspaces on which they describe the action of the Cartan Subalgebra are the \textbf{root spaces} ($\mathfrak{g}_\alpha$). One can see that
\begin{equation}
    \mathfrak{g} = \mathfrak{h} \oplus \bigoplus_{\alpha \in R} \mathfrak{g}_\alpha.\end{equation}

We consider the subgroup of the Cartan subalgebra automorphisms generated by the maps
\[
h \mapsto \mathrm{Ad}_{\exp(g)} h \quad \text{for } h \in \mathfrak{h} \text{ and } g \in \mathfrak{g}_\alpha,
\]
and call this group the \textbf{Weyl group}.

For $\su(N)$, the Cartan subalgebra is the set of traceless, purely imaginary diagonal matrices. Denoting by $L_i$ the functional that gives the $i$-th element of the diagonal, we can determine that roots have the form $L_i-L_j$. The reflection generated from a particular $\alpha$ , i.e.
\begin{equation}
s_{\alpha}(h) = h - \frac{2(h|\alpha)}{(\alpha|\alpha)}\alpha = h - \frac{2(h,L_i-L_j)}{(L_i-L_j|L_i-L_j)}(L_i-L_j) = h - (h,L_i-L_j)(L_i-L_j)
\end{equation}
acts by transposing the $i$-th and $j$-th elements of the diagonal. Here by inner product we understand the standard inner product on forms in a basis given by our choice of Cartan Subalgebra. In what follows we will define the inner product in general by the Killing form. It is a fact, that the Weyl group is generated by these reflections. Therefore, our Weyl group is the permutation      group of the diagonal entries, $S_N$.

We can consider the action of the Cartan Subalgebra on an arbitrary representation $V$. Repeating the arguments given before, we arrive at the conclusion that for an arbitrary representation $V$,
\[
V = \bigoplus_{\mu \in \mathfrak{h}^*} V_\mu
\]
where for $h \in \mathfrak{h}$ and $v \in V_\mu$, we have $h \cdot v = \mu(h)v$. The functionals $\mu$ for which $V_\mu \neq \{0\}$ are called the \textbf{weights} of the representation. Knowing the weights lets us easily deduce the action of the rest of the algebra on the representation.

It turns out that the weights of a representation exhibit a deep geometric structure. To introduce geometry into our, to this moment, purely algebraic considerations, we have to introduce a scalar product—the Killing form, defined by:
\[
K(X,Y) := \mathrm{Tr}(\mathrm{ad}(X) \circ \mathrm{ad}(Y)).
\]
The most important features of the Killing form are:
\begin{itemize}
    \item It is a symmetric, bilinear form on the Lie algebra.
    \item It is Adjoint-invariant, i.e., $K([X,Y],Z) = K(X,[Y,Z])$.
    \item It is non-degenerate if and only if the considered algebra is semi-simple.
    \item For a simple algebra, it is the unique (up to a scalar) Adjoint-invariant, symmetric bilinear form.
\end{itemize}
This last property lets us quickly determine the form of our inner product. For example, for $\su(N)$ we postulate the form given by $\mathrm{Tr}(\pi(X)\pi(Y))$, where $\pi$ is the defining (N-dimensional) representation. It is easily seen that it is symmetric, non-degenerate, and Adjoint-invariant. Thus, it is a Killing form up to some constant in which we are not interested. Throughout this article, we will denote this form, evaluated on lie algebra elements $X,Y$ by $-(X|Y)$, and the norm generated by it by $\|X\|^2 = -(X|X)$\footnote{The minus is introduced to make our form positive-definite}. We will assume that the proportionality constants are such that the form is positive definite.

General algebraic considerations yield a result that possible representation weights are those $\lambda$ which fulfill the condition
\begin{equation}
    \label{eq:intcond}
    \frac{2 (\lambda|\alpha)}{(\alpha|\alpha)} \in \mathbb{Z}
\end{equation}

for all roots $\alpha$. For $\su(N)$, where we can write $\lambda=(\lambda_1, \dots, \lambda_N)$, the condition in Eq. \eqref{eq:intcond} for $\alpha=L_i-L_j$ reads $\lambda_i - \lambda_j \in \mathbb{Z}$. Applying the additional condition that $\sum_i \lambda_i = 0$, we find that allowed weights have the form
\[
\frac{1}{N}(m_1, \dots, m_N) \quad \text{where } m_i \in \mathbb{Z} \text{ and } \sum_i m_i = 0.
\]
Irreducible representations are classified by a \textbf{highest weight} $\lambda$, which is defined as the weight maximizing some linear functional $l$ on the space of weights. We choose $l$ to be irrational with respect to the root lattice to ensure the uniqueness of the highest weight. The weights $\alpha$ that give a positive number when evaluated on $l$ are called \textbf{positive roots} and are denoted by $R_+$. We define \textbf{simple roots} as those positive roots which are not expressible as a sum of other ones. The highest weight space of an irreducible representation is always one-dimensional.

The character $\chi(V)$ of a representation is defined abstractly by $\chi(V) = \sum_\mu (\dim V_\mu)e^\mu$. Using the result that for a compact group every element $g$ is conjugate to an element $e^{ih_{alg}}$ of the maximal torus, where $i h_{alg}$\footnote{It is important to note, that $h_{alg}$ is not an element of the Cartan subalgebra $\mathfrak{g}$ - it is only when multiplied by $i$. One shall note, that we need this $i$ to introduce an isomorphism betweeen the imaginary-valued evaluations and the real-valued Killing form} is in the Cartan subalgebra $\mathfrak{g}$, we may calculate the trace of the matrix that represents $g$, also called the character of $g$:
\[
\mathrm{Tr}(\pi(g)) = \sum_\mu (\dim V_\mu)e^{i\langle \mu, h_{alg} \rangle}.
\]
The character of an irreducible representation with highest weight $\lambda$ is given by the Weyl Character Formula (WCF):
\begin{equation}
    \chi_\lambda(e^{ih_{alg}}) = \frac{\sum_{w \in W} \sgn(w) e^{i\langle w(\lambda + \rho),h_{alg} \rangle}}{\prod_{\alpha \in R_+} (e^{i\langle\alpha,h_{alg}\rangle/2} - e^{-i\langle\alpha,h_{alg}\rangle/2})},
\end{equation}
where $\rho := \frac{1}{2}\sum_{\alpha \in R_+} \alpha$ is the Weyl vector. The Killing form allows us to identify $\mathfrak{h}$ with its dual $\mathfrak{h}^*$, so we can write $\langle \mu, h_{alg} \rangle = (\mu|h)$, where $h$ is the element in $\mathfrak{h}^*$ corresponding to $h_{alg}$. Using the Weyl group invariance of the Killing form, we arrive at:

\begin{equation}
    \label{eq:wcf}
    \chi_\lambda(e^{ih}) = \frac{\sum_{w \in W} \sgn(w) e^{i(\lambda + \rho|w h)}}{\prod_{\alpha \in R_+} (e^{i(\alpha|h)/2} - e^{-i(\alpha|h)/2})}.
\end{equation}
From the WCF, we can derive a formula for the dimension of an irreducible representation by taking the limit $h \to 0$:
\begin{equation}
    \label{eq:dimension_formula}
    \dim V_{\lambda} = \prod_{\alpha \in R_+} \frac{(\lambda + \rho|\alpha)}{(\rho|\alpha)}.
\end{equation}

\section{A Motivating Example: The SU$(2)$ Limit}
\label{sec:su2_example}

To motivate our thesis, we will now consider a known\cite{dulianRandomMatrixModel2024} result that it is fulfilled for SU$(2)$ . 

We are interested in the limit of the quantity $\chi_l(\theta)/d_l$. We know that the weight space for the spin-$l$ representation of SU$(2)$ consists of the $2l+1$ states $\{ -l, -l+1, \dots, l \}$. Therefore, our character comes out to be the sum over these weights:
\begin{equation}
    \chi_j(\theta) = \sum_{m=-l}^{l} e^{im\theta}.
\end{equation}
This is a geometric series, which we can sum explicitly:
\begin{equation}
    \sum_{m=-l}^{l} (e^{i\theta})^m = e^{-il\theta} \sum_{k=0}^{2l} (e^{i\theta})^k = e^{-il\theta} \frac{e^{i\theta(2l+1)}-1}{e^{i\theta}-1}.
\end{equation}
By factoring out phases from the numerator and denominator, this simplifies to the well-known closed form:
\begin{equation}
    \chi_l(\theta) = \frac{\sin\left((l+\frac{1}{2})\theta\right)}{\sin(\frac{1}{2}\theta)}.
\end{equation}
One can see it directly just by writing down the WCF for this special case:

\begin{equation}
    \chi_\lambda(e^{ih}) = \frac{\sum_{w \in W} \sgn(w) e^{i(w(\lambda + \rho)|h)}}{ e^{i(\alpha|h)/2} - e^{-i(\alpha|h)/2} } = \frac{e^{i(l+1/2)\theta} - e^{-i(l+1/2)\theta}}{e^{i\theta/2} - e^{-i\theta/2}} = \frac{\sin\left((l+\frac{1}{2})\theta\right)}{\sin(\frac{1}{2}\theta)}.
\end{equation}

Where $h$ is an element dual to $\theta$.
For a fixed $\theta \neq 0$, we see that this value is bounded, as the sine functions oscillate. The dimension of our representation is $d_l = 2l+1$. Therefore, the normalized character behaves in the high-dimension limit ($l \to \infty$) as:
\begin{equation}
    \frac{\chi_l(\theta)}{d_l} \sim \frac{\text{Bounded Constant}}{2l+1}.
\end{equation}
This ratio clearly vanishes as $l^{-1}$. The only exception is, of course, the trivial element ($\theta = 0$), where the normalized character is always 1.

This simple case demonstrates that the normalized character vanishes in the high-dimension limit for all non-trivial group elements, a phenomenon we will investigate in section \ref{sec:analysis} for SU$(N)$ and in section \ref{sec:generalization} for general compact, reductive groups.

\section{Asymptotic Analysis at Singular Elements}
\label{sec:analysis}

The Weyl Character Formula in the form of Eq.~\eqref{eq:wcf} holds for group elements that do not lead to a zero in the denominator; we shall call such elements \textbf{regular}. For any fixed regular element, it is not hard to see that the normalized character vanishes in the high-dimension limit. The denominator of Eq. \eqref{eq:wcf} is independent of the representation $\lambda$, while the numerator is a finite sum of complex exponentials and is thus bounded. Since the dimension of the representation grows polynomially with $\lambda$, the ratio must converge to zero.

The entire difficulty of the proof, therefore, is to show that the same conclusion holds for \textbf{singular} (non-regular) elements, for which the formula becomes an indeterminate $0/0$ form. Resolving this is the primary technical challenge of this work.

\subsection{Case Study: The SU$(3$) Character at a Singular Point}
\label{sec:su3_case_study}

In this section, we provide a rigorous derivation of the character of an irreducible representation of SU$(3)$ at a singular element of its maximal torus. This serves as an illustration of the more abstract reasoning that will be provided later for the general case. By resolving the indeterminate form of the Weyl Character Formula (WCF) via a limiting procedure, we demonstrate that the character is a finite sum of terms, where each term's magnitude is proportional to the dimension of an effective SU$(2)$ sub-representation. We then prove that the normalized character vanishes in the high-dimension limit by performing an exact algebraic cancellation.

\begin{theorem}
For any fixed singular element $g \neq I$ in SU$(3)$, corresponding to a Cartan element $h_0 \neq 0$, the normalized character vanishes in the high-dimension limit:
\[
\lim_{\lambda \to \infty} \frac{\chi_\lambda(g)}{\dimL} = 0.
\]
\end{theorem}

\begin{proof}

We begin with the Weyl Character Formula, Eq. \eqref{eq:wcf}. This formula holds for any regular element $h$. We will denote the positive roots of SU$(3)$ by $\alpha_1 = L_1 - L_2$, $\alpha_2 = L_2 - L_3$, and $\alpha_1 + \alpha_2 = L_1 - L_3$. The Weyl group is the permutation group of three elements, $S_3$. The Weyl vector is $\rho_{\su(3)} = \alpha_1 + \alpha_2 = (1,0,-1)$. $\alpha_1$ and $\alpha_2$ are the simple roots.

We now address the degenerate case, where the denominator is zero. This occurs when $h$ lies in a hyperplane orthogonal to some root $\alpha$. We consider a singular element $h_0$ such that $(\alpha_1|h_0) = 0$ but $(\alpha|h_0) \neq 0$ for $\alpha \in \{\alpha_2, \alpha_1+\alpha_2\}$. As our $\alpha_1 = L_1 - L_2$, this means exactly that $h_0$ has the following form:
\[
h_0 = \begin{pmatrix}
    a & 0 & 0 \\
    0 & a & 0 \\
    0 & 0 & b
\end{pmatrix},
\]
where $2a+b=0$ and $a \neq b$.

To evaluate the character at $h_0$, we consider a nearby regular point $h = h_0 + \delta$ and take the limit $\delta \to 0$. The WCF for $h$ reads:
\begin{equation}
    \chi_\lambda(e^{i(h_0+\delta)}) = \frac{1}{\prod_{\alpha \in \{\alpha_2, \alpha_1+\alpha_2\}} \left(e^{i(\alpha|h_0+\delta)/2} - e^{-i(\alpha|h_0+\delta)/2}\right)} \cdot \frac{\sum_{w \in W} \sgn(w) e^{i(\eta|w(h_0+\delta))}}{2i\sin\left(\frac{1}{2}(\alpha_1|\delta)\right)}
\end{equation}

where we have introduced $\eta := \lambda + \rho_{\su(3)}$. We can understand this degeneracy by looking at the action of the Weyl group. The degeneracy with respect to $\alpha_1$ means that the transposition of the first and the second element keeps $h_0$ the same. We can look at it from a geometrical point of view. The reflection of $h_0$ in the plane perpendicular to $\alpha_1$ is given by
\[
s_{\alpha_1}(h_0) = h_0 - \frac{2(h_0|\alpha_1)}{(\alpha_1|\alpha_1)}\alpha_1.
\]
As $(h_0|\alpha_1)$ is 0, we see that $h_0$ is a fixed point of this reflection. The subgroup of the Weyl group that leaves $h_0$ invariant is the stabilizer, $W_0 = \{e, s_{\alpha_1}\}$. We can isolate this stability by factoring the numerator sum over the cosets $b \in W/W_0$. A set of coset representatives is $\{e, s_{\alpha_1}, s_{\alpha_2}s_{\alpha_1}\}$. The factorization is:
\begin{equation}
    \sum_{w \in W} \sgn(w) e^{i(\eta|w(h_0+\delta))} = \sum_{b \in W/W_0} \sgn(b) e^{i(\eta|bh_0)} \left( \sum_{\sigma \in W_0} \sgn(\sigma) e^{i(\eta|b\sigma\delta)} \right).
\end{equation}
We analyze the singular part of the expression for each term $b$ by conjugating the stabilizer:
\begin{equation}
    \sum_{\sigma \in W_0} \sgn(\sigma) e^{i(\eta|b\sigma\delta)} = \sum_{\sigma \in bW_0b^{-1}} \sgn(\sigma) e^{i(\eta|\sigma b\delta)} = \sum_{\sigma \in \{e, s_{b\alpha_1}\}} \sgn(\sigma) e^{i(\eta|\sigma b\delta)}.
\end{equation}
If we assume that $\delta\in \text{span}(\alpha_1)$, we can interpret the ratio of this sum to the singular part of the denominator as a WCF for an SU$(2)$ sub-algebra given as a subspace $\mathfrak{g}_{b\alpha_1} \oplus \mathfrak{g}_{-b\alpha_1} \oplus [\mathfrak{g}_{b\alpha_1},\mathfrak{g}_{-b\alpha_1}] $  associated with the root $b\alpha_1$:
\begin{equation}
    \frac{\sum_{\sigma \in \{e, s_{b\alpha_1}\}} \sgn(\sigma) e^{i(\lambda + \rho_{\mathfrak{su}(3)}|\sigma b\delta)}}{2i\sin\left(\frac{1}{2}(b\alpha_1|\delta)\right)} = \frac{\sum_{\sigma \in \{e, s_{b\alpha_1}\}} \sgn(\sigma) e^{(\lambda'_b+\rho_{\su(2)}|\sigma b\delta)}}{2i\sin\left(\frac{1}{2}(b\alpha_1|\delta)\right)}
\end{equation}
where both exponents have to be equal, $\lambda'_b \in \text{span} (b\alpha_1)$ and $\rho_{\su(2)}=\frac{1}{2}b\alpha_1$. Thus the effective highest weight $\lambda'_{b}$\footnote{The fact that $\lambda'_b$ is a valid weight follows from the construction of the sub-representation, and will be proven explicitly in general in the next section in the lemma \ref{lemma:weights}.} of this sub-representation is defined by the condition:

\begin{equation}
    \label{eq:effective_weight}
    \lambda'_{b} + \rho_{\su(2)} = \frac{(\lambda + \rho_{\su(3)}|b\alpha_1)}{(b\alpha_1|b\alpha_1)} b\alpha_1.
\end{equation}
 In the limit $\delta \to 0$, this ratio gives the dimension of this sub-representation, $\dim L_{\lambda'_{b}}$:
\begin{equation}
    \frac{\sum_{\sigma \in \{e, s_{b\alpha_1}\}} \sgn(\sigma) e^{i(\eta|\sigma b\delta)}}{2i\sin\left(\frac{1}{2}(b\alpha_1|\delta)\right)} \xrightarrow{\delta \to 0} \dim L_{\lambda'_{b}}.
\end{equation}

Taking the limit of the full expression, we obtain the exact character at $h_0$:
\begin{equation}
    \chi_\lambda(h_0) = \frac{1}{\prod_{\alpha \in \{\alpha_2, \alpha_1+\alpha_2\}} \left(e^{i(\alpha|h_0)/2} - e^{-i(\alpha|h_0)/2}\right)} \sum_{b \in W/W_0} \sgn(b) e^{i(\eta|b h_0)} \dim L_{\lambda'_{b}}.
\end{equation}
To analyze the limit $||\lambda|| \to \infty$, we consider the ratio $|\chi_\lambda(h_0)|/\dimL$. By the triangle inequality, this is bounded by a finite sum of positive terms. The limit of this sum is zero if the limit of each term's ratio is zero. We therefore analyze the ratio for a representative term, $b=e$:

From Eq. (\ref{eq:dimension_formula}) we know that the exact dimension of the sub-representation is 
\begin{equation}
    \mathrm{dim}\,L_{\lambda'_{e}} = \frac{(\lambda' + \rho_{\su(2)}|\alpha_1)}{(\rho_{\su(2)}|\alpha_1)}= \frac{(\lambda + \rho_{\su(3)}|\alpha_1)}{(\rho_{\su(2)}|\alpha_1)}
\end{equation}

where we have used \ref{eq:effective_weight}. The ratio of this dimension to that of the full representation is:

\begin{align}
    \frac{\mathrm{dim}\,L_{\lambda'_{e}}}{\mathrm{dim}\,L_\lambda} 
    &= \frac{(\lambda + \rho_{\su(3)}|\alpha_1)}{(\rho_{\su(2)}|\alpha_1)} \left( \prod_{\alpha \in R_+} \frac{(\lambda+\rho_{\su(3)}|\alpha)}{(\rho_{\su(3)}|\alpha)} \right)^{-1} \nonumber \\[\jot]
    &= \frac{(\rho_{\su(3)}|\alpha_1)}{(\rho_{\su(2)}|\alpha_1)} \cdot \frac{(\rho_{\su(3)}|\alpha_2)(\rho_{\su(3)}|\alpha_1+\alpha_2)}{(\lambda+\rho_{\su(3)}|\alpha_2)(\lambda+\rho_{\su(3)}|\alpha_1+\alpha_2)}. \label{eq:ratio_limit_rigorous}
\end{align}

The numerator of this final expression in Eq. \eqref{eq:ratio_limit_rigorous} is a constant that depends only on the fixed Weyl 
vectors $\rho_{\su(3)}$ and $\rho_{\su(2)}$. The denominator is a product of terms that grow linearly with $\lambda$. 
Therefore, the ratio vanishes as $\lambda \to \infty$. Since every term in the bounding sum 
vanishes, the proof is complete.

\end{proof}

\subsection{The General Case: SU$(N)$ Characters}
\label{sec:sun_proof}

We now extend our analysis to the general case of SU$(N)$. The first modification with respect to the SU$(3)$ case is that our stabilizer is a subgroup generated by reflections in all degenerate roots. Let $h_0$ be a singular element in the Cartan subalgebra, let $R_{+, \text{deg}}$ be the set of positive roots orthogonal to $h_0$, and let $R_{+, \text{ndeg}}$ be the set of non-degenerate positive roots. We evaluate the character at $h = h_0 + \delta$ and take the limit $\delta \to 0$. The Weyl Character Formula separates into regular and singular parts:
\begin{equation}
    \chi_\lambda(e^{ih}) = \frac{1}{\prod_{\alpha \in R_{+, \text{ndeg}}} \left(e^{i(\alpha|h_0)/2} - e^{-i(\alpha|h_0)/2}\right)} \cdot \frac{\sum_{w \in W} \sgn(w) e^{i(\eta|w(h_0+\delta))}}{\prod_{\alpha \in R_{+, \text{deg}}} \left(e^{i(\alpha|\delta)/2} - e^{-i(\alpha|\delta)/2}\right)},
\end{equation}
where $\eta := \lambda + \rho_{\su(N)}$. As for the SU$(3)$ case, we factorize our sum over the stabilizer $W_0 = \Stab_W(h_0)$:
\begin{equation}
    \sum_{w \in W} \sgn(w) e^{i(w\eta|h_0+\delta)} = \sum_{b \in W/W_0} \sgn(b) e^{i(\eta|b h_0)} \left( \sum_{\sigma \in W_0} \sgn(\sigma) e^{i(\eta|b\sigma \delta)} \right).
\end{equation}
As in \ref{sec:su3_case_study}, we consider only the singular part of our formula for each coset representative $b$. After conjugating the stabilizer, this part is:
\begin{equation}
    \label{eq:effchar}
    \frac{\sum_{\sigma \in bW_0b^{-1}} \sgn(\sigma) e^{i(\eta|\sigma b\delta)}}{\prod_{\alpha \in bR_{+, \text{deg}}} \left(e^{i(\alpha|b\delta)/2} - e^{-i(\alpha|b\delta)/2}\right)}.
\end{equation}
To interpret Eq. \eqref{eq:effchar} as the WCF for an effective representation of a subgroup whose root system is $bR_{\text{deg}}$ we have to assume that $\delta$ in the span of the degenerate roots and need the following

\begin{lemma}
\label{lemma:weights}
An effective weight $\lambda'$, which fulfills
$$
\lambda' + \rho_b' = \pi_{bR_{deg}}(\lambda + \rho),
$$
where $\lambda$ is an integral weight, $\pi_{bR_{deg}}$ is the orthogonal projection onto the subspace spanned by the effective roots $bR_{deg}$, and $\rho_b'$ is the Weyl vector corresponding to the effective representation, belongs to the weight lattice of the effective representation.
\end{lemma}

\begin{proof}
The condition we need to check is that $\frac{2(\alpha|\lambda')}{(\alpha|\alpha)}\in\mathbb{Z}$ for all simple roots $\alpha \in bR_{+,deg}$. We start with the definition of $\lambda'$. Because $\alpha$ lies in the subspace on which we project, by the properties of the orthogonal projection we have $(\alpha|\lambda' + \rho_b') = (\alpha|\lambda + \rho)$.

This directly implies:
$$
\frac{2(\alpha|\lambda'+\rho_b')}{(\alpha|\alpha)}=\frac{2(\alpha|\lambda+\rho)}{(\alpha|\alpha)}
$$
Rearranging this equation to solve for the term of interest, we get:
\begin{equation}
    \label{eq:integrality}
    \frac{2(\alpha|\lambda')}{(\alpha|\alpha)}=\frac{2(\alpha|\lambda)}{(\alpha|\alpha)}+\frac{2(\alpha|\rho)}{(\alpha|\alpha)}-\frac{2(\alpha|\rho_b')}{(\alpha|\alpha)}
\end{equation}
We analyze each term on the right-hand side of Eq. \eqref{eq:integrality}:
\begin{enumerate}
    \item The term $\frac{2(\alpha|\lambda)}{(\alpha|\alpha)}$ is an integer because $\lambda$ is an integral weight of the original representation and $\alpha$ is a root within that system.
    \item The term $\frac{2(\alpha|\rho)}{(\alpha|\alpha)}$ is an integer because the Weyl vector $\rho$ is a known integral weight of the original representation.
    \item The term $\frac{2(\alpha|\rho_b')}{(\alpha|\alpha)}$ is equal to $1$, by the fundamental property of the Weyl vector $\rho_b'$, since $\alpha$ is a simple root of the effective system.
\end{enumerate}

All the terms on the right of Eq. \eqref{eq:integrality} are integers, so their sum is an integer. Therefore, the integrality condition is fulfilled, and $\lambda'$ belongs to the weight lattice of our effective representation. (It may not belong to the weight lattice of our original representation, as we are checking the integrality condition only for roots in $bR_{+,deg}$.)
\end{proof}

In the limit $\delta \to 0$, this ratio gives the dimension of this sub-representation. By bounding the magnitude of the full character at $h_0$ and analyzing the limit of a representative term, we find that the ratio $|\chi_\lambda(h_0)|/\dimL$ is controlled by the ratio of dimension formulas. This is the same procedure as for su$(3)$. For a representative term, this ratio is:
\begin{equation}
    \label{eq:boundconst}
    \frac{\prod_{\alpha \in R_{+, \text{deg}}} \frac{(\lambda+\rho|\alpha)}{(\rho'_{b}|\alpha)}}{\prod_{\alpha \in R_{+}} \frac{(\lambda+\rho|\alpha)}{(\rho|\alpha)}} = \frac{\text{Constant}}{\prod_{\alpha \in R_{+, \text{ndeg}}} (\lambda+\rho|\alpha)}.
\end{equation}
To finish, we just have to prove the following

\begin{lemma}
\label{lemma:divergence}
For any non-zero singular element $h_0$ and a highest weight $\lambda$, the product over the non-degenerate roots diverges as $||\lambda|| \to \infty$:
\[
\lim_{||\lambda|| \to \infty} \prod_{\alpha \in R_{+, \text{ndeg}}} (\lambda+\rho|\alpha) = \pm \infty.
\]
\end{lemma}

\begin{proof}
By Eq (\ref{eq:intcond}) we know that $\frac{2 (\lambda|\alpha)}{(\alpha|\alpha)} \in \mathbb{Z}$, and thus if one of our product's factors is diverging, the whole product is diverging too. We will show that at least one of the factors is bigger or equal to $\frac{1}{2}|\lambda|_{\infty}$ which goes to infinity as $||\lambda||\to\infty$. We will argue by contradiction, i.e. we will assume that for every $\alpha \in R_{+, \text{ndeg}}$, we have $|(\lambda|\alpha)| < \frac{1}{2}|\lambda|_\infty$.

Consider two diagonal entries $(h_0)_i$ and $(h_0)_j$ of our singular element $h_0$ (from Eq. \ref{eq:effchar}). We assume that $h_0$ is nontrivial, i.e. not all its diagonal entries are the same. In what follows we WLOG assume that $i>j>k$. Now we have two cases:

\begin{enumerate}

    \item If $(h_0)_i=(h_0)_j$, then we can find an element $(h_0)_k$ such that $(h_0)_k \neq (h_0)_j$, because $h_0$ is nontrivial. This means that both roots $\alpha_{jk}=L_{j}-L_{k}$ and $\alpha_{ik}=L_{i}-L_{k}$ belong to $R_{+,ndeg}$. Therefore, by the triangle inequality,
    $$
    |(\lambda)_{i}-(\lambda)_{j}| = |((\lambda)_i - (\lambda)_k) - ((\lambda)_j - (\lambda)_k)| \le |(\lambda)_i - (\lambda)_k| + |(\lambda)_j - (\lambda)_k| < \frac{1}{2}|\lambda|_\infty + \frac{1}{2}|\lambda|_\infty = |\lambda|_\infty.
    $$

    \item If $(h_0)_i \neq (h_0)_j$, then the root $\alpha_{ij} = L_i - L_j$ is in $R_{+,ndeg}$. Thus, by our assumption, $|(\lambda)_{j}-(\lambda)_{i}| < \frac{1}{2}|\lambda|_{\infty} < |\lambda|_\infty$.

\end{enumerate}

We WLG assume that the maximum-module component of $\lambda$ is $\lambda_{1} = |\lambda|_{\infty} > 0$. From the two cases considered, we know that for every $j$, $|\lambda_{1}-\lambda_{j}| < |\lambda|_{\infty}$. Since $\lambda_1 = |\lambda|_\infty$, this gives $\lambda_{1}-\lambda_{j} \le |\lambda_{1}-\lambda_{j}| < \lambda_{1}$. The inequality $\lambda_{1}-\lambda_{j} < \lambda_{1}$ implies that all of the components $\lambda_j$ are positive, which is not possible as $\sum_{j}\lambda_{j}=0$.

\end{proof}

\begin{corollary}
    The rate of convergence of the normalized trace is entirely dependent on the path $\lambda_i$ by which we go to $\infty$. However, regardless of the path, we have a universal upper bound proportional to $1/|\lambda|_{\infty}$. For specific paths, the rate of convergence can be easily deduced from \ref{eq:boundconst}, for example one can easily see that for a path $n\lambda$ where $n\in \mathbb{Z}$ the rate is $1/n^m$ where m is the number of nondegenerate roots, that are not perpendicular to $\lambda$ i.e. $(\lambda,\alpha)\neq0$
\end{corollary}

A deep analysis of the rate of convergence from the topological point of view, can be found in \cite{etingof2025boundsasymptoticcharacterssimple}.

\section{Discussion and Implications}
\label{sec:discussion}

\subsection{Interpretation: The Stabilizer and the Centralizer}

The formula for the character at a singular point, Eq.~\eqref{eq:effchar}, has a deep geometric meaning that relates the structure of the stabilizer subgroup $W_0$ to the centralizer of $h_0$. We will show that a slight change in our limiting procedure results in the WCF for the centralizer subgroup of $h_0$.

\begin{lemma}
The centralizer of a singular element $h_0 \in \su(N)$ is a product of smaller unitary groups. Specifically, if $h_0$ has degenerate eigenvalues with multiplicities $n_1, n_2, \dots, n_k$, then its centralizer is $Z(h_0) \cong \mathrm{S}(\mathrm{U}(n_1) \times \dots \times \mathrm{U}(n_k)) \cong \su(n_1) \times \dots \times \su(n_k) \times \mathfrak{u}(1)^{k-1}$.
\end{lemma}

To see why the isomorphism holds we consider the SU$(5)$ case and $h_0=\text{diag}(a,a,a,b,b)$. The centralizer will turn out to be the set of block diagonal matrices 

\[
  M =
  \begin{pmatrix}
    \mathbf{A} & \mathbf{0} \\
    \mathbf{0} & \mathbf{B}
  \end{pmatrix},
\]
where $\mathbf{A}$ is a $3 \times 3$ unitary matrix and $\mathbf{B}$ is a $2 \times 2$ unitary matrix. The overall determinant must be equal to one, so we can parametrize our matrix by two special unitary matrices and a relative phase between them, giving the group structure $\mathrm{S(U(3) \times U(2))} \cong \mathrm{SU}(3) \times \mathrm{SU}(2) \times \mathrm{U}(1)$. 

\begin{proof}
We consider the one dimensional lie algebra containing $h_0$. $Mh_0=h_0M$, therefore $M$ is an interwiner of the defining representations. From the Schur lemma we know that those can act only between the isomorphic representations and for an abelian lie algebra representations are isomorphic iff they have the same eigenvalue. Thus we conclude the block diagonal form of M.
\end{proof}

\begin{lemma}
The stabilizer subgroup $W_0 = \Stab_W(h_0)$ is generated by the reflections in the degenerate roots. It is the permutation group on the components of $h_0$ that leaves it invariant.
\end{lemma}
\begin{proof}
The fact that $\alpha(h_0)=0$ for a root $\alpha=L_i-L_j$ is formally equivalent to saying that for $h_0 = \text{diag}(h_1, \dots, h_N)$, the components $h_i$ and $h_j$ are equal. The reflection $s_\alpha$ acts by transposing the components $h_i$ and $h_j$. The group generated by these transpositions is precisely the permutation group that preserves the block structure of equal diagonal entries.
\end{proof}

Going back to our SU$(5)$ example, we see that the Weyl group of the centralizer is $W(Z(h_0)) = W(\mathrm{SU}(3)) \times W(\mathrm{SU}(2)) = S_3 \times S_2$. This is exactly the stabilizer of $h_0$. But we cannot yet conclude that the effective group whose character we found in Eq.~\eqref{eq:effchar} is the full centralizer, as the Weyl group does not uniquely specify the group. Indeed, the limiting procedure with $\delta$ constrained to the span of degenerate roots recovers the WCF for the semi-simple part of the centralizer, e.g., $\mathrm{SU}(3) \times \mathrm{SU}(2)$. The abelian $\mathrm{U}(1)$ factors, which have trivial Weyl groups, are recovered from the phase factors that emerge when $\delta$ is taken to be generic.

\begin{lemma}
The group that we get in the limiting procedure is the semisimple part of the centralizer of $h_0$.
\end{lemma}
\begin{proof}
We define our effective group algebra as an algebra generated by a direct sum of root spaces corresponding to degenerate roots, $\mathfrak{g}_{\text{eff}} = \langle\bigoplus_{\alpha \in R_{\text{deg}}} \mathfrak{g}_\alpha\rangle$. For any $k \in \mathfrak{g}_{\text{eff}}$, it is obvious that $[k, h_0] = \sum_{\alpha \in R_{\text{deg}}} [k_\alpha, h_0] = \sum_{\alpha \in R_{\text{deg}}} \alpha(h_0)k_\alpha = 0$, so the exponents commute. One can easily see that the Killing form restricted to this algebra is nondegenerate.

\end{proof}

Now, we have to express what we mean by the "almost" part. It's easy to see that the full centralizer algebra is given by $\mathfrak{h} \oplus \bigoplus_{\alpha \in R_{\text{deg}}} \mathfrak{g}_\alpha$. But one can see that not a whole $\mathfrak{h}$ is generated by the commutators within $\mathfrak{g}_{\text{eff}}$. Indeed, the space spanned by the Cartan elements $[\mathfrak{g}_{\alpha},\mathfrak{g}_{-\alpha}] $ for the non-degenerate roots $\alpha$ is absent.

It seems strange that we get only part of the centralizer in our procedure. The key to all of this is that we have assumed that $\delta \in \mathrm{span}(R_{+, \text{deg}})$, which simplified our reasoning. Now we relax this assumption and write $\delta = \delta_\parallel + \delta_\perp$, where these terms correspond respectively to the part that belongs to the span of the degenerate roots and the part that is orthogonal to this space. The second part, $\delta_\perp$, is an invariant of the stabilizer group $W_0$. The term $e^{i(\eta|b\sigma\delta_\perp)}$ can be simplified to $e^{i(\eta|b \delta_\perp)}$ and factored out of the inner sum. Thus we get:
\begin{equation}
    e^{i(\eta|b\delta_\perp)} \sum_{\sigma \in W_0} \sgn(\sigma) e^{i(\eta|b\sigma\delta_\parallel)}.
\end{equation}
One sees that the factor before our sum gives us exactly the missing character of the U(1) factors that form the center of the centralizer. Thus, we recover the character of the full centralizer of $h_0$ and prove the

\begin{corollary}
    The group that we get in the limiting procedure when we do not restrict $\delta$ is exactly the centralizer of $h_0$.
\end{corollary}

\section{Generalizations and Limitations}
\label{sec:generalization}

The arguments from section \ref{sec:sun_proof} leading to the final form of the character at a singular point are general manipulations of the WCF, so they transfer verbatim to the other classical compact simple Lie groups. The crux of the generalization is to establish a version of Lemma \ref{lemma:divergence} for each classical series. As the proof must employ the specific root structure of each Lie algebra, it should be conducted for each type separately. Yet, it is possible to give a unified argument for all classical series using the connectivity of the Dynkin diagram, as we will show in \ref{sec:exceptional} while discussing the exceptional algebras.

\subsection{Classical Simple Algebras}

We can assume that we work in the complexified Lie algebra, as a bound for it will automatically bound the compact real form. This lets us assume the diagonal form of the Cartan Subalgebra. As in the SU$(N)$ case, we assume by contradiction that the product from Lemma \ref{lemma:divergence} remains bounded as $\lambda \to \infty$.

\subsubsection{Type $B_n$: The $\mathfrak{so}(2n+1)$ Algebra}
The set of positive roots is $R_+ = \{ L_i \pm L_j \mid 1 \le i < j \le n \} \cup \{ L_i \mid 1 \le i \le n \}$. Since our singular element $h_0 \neq 0$, at least one of its components, say $(h_0)_i$, is non-zero. This implies that the corresponding root $L_i$ must be non-degenerate, i.e., $L_i \in R_{+, \text{ndeg}}$. For the product to be bounded, the term $(\lambda+\rho, L_i)$ must remain bounded. Since $\rho$ is fixed, this means the component $\lambda_i$ must be bounded.

Now consider any other index $j \neq i$. The root $L_i-L_j$ or $L_i+L_j$ must be non-degenerate (otherwise $h_i = h_j = 0$ for all $j$, which implies $h_0=0$). For the product to be bounded, one of $(\lambda+\rho, L_i-L_j) \approx \lambda_i - \lambda_j$ and $(\lambda+\rho, L_i+L_j) \approx \lambda_i + \lambda_j$ must be bounded. Since $\lambda_i$ is bounded, this implies that $\lambda_j$ must also be bounded. As this holds for all $j$, the norm of $\lambda$ remains bounded, which contradicts the high-dimension limit.

\subsubsection{Type $C_n$: The $\mathfrak{sp}(n)$ Algebra}
The positive roots are $R_+ = \{ L_i \pm L_j \mid 1 \le i < j \le n \} \cup \{ 2L_i \mid 1 \le i \le n \}$. It is easy to see that the proof is identical to the $B_n$ case, with the root $L_i$ replaced by $2L_i$.

\subsection*{Type $D_n$}
The positive roots are $R_+ = \{ L_i \pm L_j \mid 1 \le i < j \le n \}$. First we will consider 

\subsubsection*{$n>2$ case}
We WLOG assume that $(h_0)_1\neq0$. We will assume the product is bounded by $d$ and we will show that all of the components of $\lambda$ are bounded.  

First, notice, that it is enough to show that $\lambda_1$ is bounded. Indeed, for every $j>1$ we have $L_1+L_j\in R_{+,ndeg}$ or $L_1-L_j\in R_{+,ndeg}$. Therefore, if $\lambda_1$ is bounded, then $\lambda_j$ is bounded as well. If we have $i$ such that $(h_0)_i\neq(h_0)_1$ and $(h_0)_i\neq-(h_0)_1$ then $$
|2\lambda_1|\leq |\lambda_1-\lambda_i|+|\lambda_1+\lambda_i|\leq 2d
$$

If not, then all of the diagonal entries are equal to $x$ or $-x$. As $h_0\neq0$, we have $x\neq0$. The easiest way to bound $\lambda_1$ is to consider two cases, first, in which our matrix diagonal is $(x,x,x,\dots)$ and second, in which it is $(x,x,-x,\dots)$. For the first case, we have $L_1+L_2$, $L_2+L_3$, $L_1+L_3\in R_{+,ndeg}$ and thus $$
|2\lambda_1|\leq |\lambda_1+\lambda_2|+|-\lambda_2-\lambda_3|+|\lambda_1+\lambda_3| \leq 3d
$$
For the second case, we have $L_1+L_2$, $L_1-L_3$, $L_2-L_3\in R_{+,ndeg}$ therefore $$
|2\lambda_1|\leq |\lambda_1+\lambda_2|+|-\lambda_2+\lambda_3|+|\lambda_1-\lambda_3| \leq 3d
$$
We bounded $\lambda_1$ and thus all of the components of $\lambda$ are bounded, which contradicts the high-dimension limit.

\subsubsection*{$n=2$ case}

Consider $L_1+L_2$ to be degenerate. Then the only nondegenerate positive root is $L_1-L_2$. We take $\lambda = (m,m)$. We readily see that the denominator is fixed as $m\to\infty$. Therefore our theorem does not hold. The reason for this is that $D_2$ is not simple. Both this and $n=1$ case will be held by the lemma in the section \ref{sec:non-simple}.

\subsection{Exceptional Simple Algebras}
\label{sec:exceptional}

Proofs for classical lie algebras relied on the convenient $L_i$ basis, which would be highly inconvenient here. Instead, we can use a more abstract argument, that uses the connectivity of the Dynkin diagram.

We need two foundational facts:
\begin{itemize}
    \item Simple roots form a dual basis for a Cartan subalgebra. 
    \item For a simple Lie algebra, Dynkin diagram is connected. It means, that for any two simple roots $\alpha_1$ and $\alpha_n$ we have a chain of simple roots $\alpha_1,\alpha_2,\dots,\alpha_n$ such that $(\alpha_i,\alpha_{i+1})<0$.
\end{itemize}
We can assume that a certain subset of simple roots spans the $R_{+,deg}$. We have $$   
 \lambda = \sum_{\alpha\in R_{simple}} \frac{(\alpha,\lambda)}{(\alpha,\alpha)} \alpha $$
We will argue by contrdiciton. We assume that we have a sequence of weights with diverging norm, for which the denominator in \ref{eq:boundconst} i.e. 
\begin{equation}
    \label{eq:thedenominator}
        \prod_{\alpha \in R_{+, \text{ndeg}}} (\lambda+\rho|\alpha) 
\end{equation}
does not diverge to $\infty$. It would then contain a bounded subsequence, thus we can WLOG assume it is bounded.

As the norm of $\lambda$ is diverging, for some simple root $\alpha$, $(\alpha,\lambda)\to \infty$. If $\alpha \in R_{+,ndeg}$ then (\ref{eq:thedenominator})$\to \infty$. If not, we take two simple roots $\alpha$ and $\beta$ degenerate and nondegenerate such that $(\alpha,\lambda)\to \infty$ and $(\beta,\lambda)$ is bounded (we only need a diverging subsequence, so we can assume that everything that do not diverge to $\infty$ is bounded, by choosing a proper subsequence). From the fact that the diagram is connected, we know that there is a path between those roots. Now, we impose that we take roots that fulfill our conditions and minimize the path length. Let us call this path $\alpha,\alpha_1,\dots, \alpha_k, \beta$. From the minimality of the path we know that every $(\alpha_i,\beta)$ is bounded.

\begin{lemma}
If $\alpha$, $\beta$ are roots, then for a string
$$ \beta-p\alpha, \beta-(p-1)\alpha, \dots, \beta+q\alpha $$
$p-q=n_{\beta,\alpha}$.
\end{lemma}

\begin{proof}
When we reflect through the hyperplane orthogonal to $\alpha$, we have to move $\beta-p\alpha$ into $\beta+q\alpha$. Thus
$$ W_{\alpha}(\beta-p\alpha) = \beta-n_{\beta,\alpha}\alpha+p\alpha=\beta+q\alpha $$
and $p-q=n_{\beta,\alpha}$. 
\end{proof}

\begin{lemma}
For a chain of roots in a Dynkin diagram $\dots-\alpha_{1} -\alpha_{2}-\dots-\alpha_{n}-\dots$, the sum of $\alpha_{i}$ is a (positive) root.
\end{lemma}

\begin{proof}
Will proceed by induction. Base case $n=1$ is trivial.

Assume $\alpha_{1}+\dots+\alpha_{k-1}$ is a root $\beta$. Notice that $(\beta,\alpha_{k})=(\alpha_{1}+\dots+\alpha_{k-1},\alpha_{k})=(\alpha_{k-1},\alpha_{k})<0$. Therefore the string generated by $\alpha_{k}$ through $\beta$ has a nonzero length. But $\alpha_{1}+\dots+\alpha_{k-1}-\alpha_{k}$ is not a root as it is a combination of the simple roots with mixed signs. Thus $\alpha_{1}+\dots+\alpha_{k}$ is a root.
\end{proof}

From this we conclude, that  $\alpha+\alpha_1+\dots+\alpha_k+ \beta \in R_{+,ndeg}$ as it does not lie in a span of degenerate simple roots. We see that
$$
    (\alpha+\alpha_1+\dots+\alpha_k+ \beta,\lambda)\to \infty
$$ 
Which proves that the denominator goes to $\infty$.

\begin{corollary}
    For any simple compact Lie algebra $\mathfrak{g}$, the normalized character of any irreducible representation $\pi_\lambda$ at a singular element $h_0 \in \mathfrak{g}$ converges to zero as the highest weight $\lambda$ goes to infinity.
\end{corollary}

\subsection{Non-Simple Reductive Algebras}
\label{sec:non-simple}
The theorem does not hold for non-simple Lie algebras, as demonstrated by the case of $D_2 \cong \su(2) \oplus \su(2)$. 

The failure for $D_2$ (and the abelian case $D_1$) is a specific instance of a general principle, which we state in the following lemma.

\begin{lemma}
For any reductive, non-simple Lie algebra $\mathfrak{g}$, one can find a non-trivial element $g$ and a sequence of representations $(\pi_k)$ with diverging dimension such that the normalized character does not converge to 0.
\end{lemma}
\begin{proof}
    For any reductive Lie algebra $\mathfrak{g}$, we have the decomposition
    \[
        \mathfrak{g} = \mathfrak{g}_{ss} \oplus \mathfrak{g}_{ab},
    \]
    where $\mathfrak{g}_{ss}$ is the semisimple part and $\mathfrak{g}_{ab}$ is the abelian part (the center). Irreducible representations of the abelian part are one- or two-dimensional, so it doesn't make sense to even talk about a high-dimension limit for that factor alone. We can therefore take a high-dimension limit if and only if the semisimple part is non-zero.

    First, we will discuss a case when non-abelian part is nonzero. We can consider a sequence of representations
    \[
        V_k = V_{ab} \otimes V_{ss,k},
    \]
    where $V_{ab}$ is an arbitrary one- or two-dimensional non-trivial irreducible representation of the abelian part, and $V_{ss,k}$ is a sequence of irreducible representations of the semisimple part whose dimension $d_k$ goes to infinity. The character of our representation evaluated at an element $g_{ab} = \exp(X_{ab})$ with $X_{ab} \in \mathfrak{g}_{ab}$ is:
    \[
        \chi_{V_k}(g_{ab}) = \chi_{V_{ab}}(g_{ab}) \cdot \chi_{V_{ss,k}}(I) = \chi_{V_{ab}}(g_{ab}) \cdot d_k.
    \]
    The dimension of the full representation is $d_{V_k} = \dim(V_{ab}) \cdot d_k$. Therefore, the normalized character is the constant $\chi_{V_{ab}}(g_{ab})/\dim(V_{ab})$, which can always be chosen to be non-zero.

    If the abelian part is zero, then the semisimple part consist of at least two simple algebras. We can take a sequence of representations that are non-trivial on only one of the simple factors, and the argument proceeds as above.
\end{proof}

The above construction, shows us exactly how the theorem fails for the reductive algebras. We can use this insight to formulate the following corollary


\begin{theorem}
\label{cor:critdiv}
Let $G=G_1\times\ldots\times G_n$ be a decomposition of a  semisimple compact Lie group into simple components $G_i$'s. Let $\lambda=(\lambda^{(1)}, \dots, \lambda^{(n)})$ be the highest weight of an irrep of $G$, where $\lambda^{(i)}$ is the highest weight of an irrep of $G_i$. Then for any fixed $g\in G$ with $g\neq e$,
\begin{equation}
\frac{\chi_\lambda(g)}{d_\lambda}\;\longrightarrow\;0 \quad 
\text{as } \forall i\quad d_{\lambda^{(i)}}\rightarrow\infty.
\end{equation}
\end{theorem}





\section{Application to spectral measure of averaging operators in large irreps}
\label{sec:quantumpart}

Let $G \cong G_1 \times \cdots \times G_n$ be a compact, semisimple Lie group and let 
$\pi_\lambda: G \to U(V_\lambda)$ be its finite-dimensional irreducible unitary representations with highest weight $\lambda$ and 
dimension $d_\lambda=\dim V_\lambda$. It is well known that 
there exist irreducible representations $\pi_{\lambda^{(i)}}:G_i\rightarrow U(V_{\lambda^{(i)}})$ with highest weights $\lambda^{(i)}$ such that
\[
\pi_\lambda \cong \pi_{\lambda^{(1)}} \otimes \cdots \otimes \pi_{\lambda^{(n)}},
\] 
and $d_\lambda=\Pi_{i=1}^n d_{\lambda^{(i)}}$. Fix a finite symmetric\footnote{$\mathcal{S}$ is symmetric iff: $g\in \mathcal{S}\iff g^{-1}\in \mathcal{S}$.} set $\mathcal{S}\subset G$ and assume that it  generates a free dense subgroup $\langle \mathcal{S}\rangle$ of $G$. 
Let $\nu$ be a uniform probability measure supported on $\mathcal{S}$, that is $\forall g\in \mathcal{S}$ we have $\nu(g)=\frac{1}{|\mathcal{S}|}$. 
We consider the averaging operator
\begin{equation}
T_{\nu,\lambda} = \sum_{g\in\mathcal{S}} \nu(g)\, \pi_\lambda(g).
\end{equation}
As $\nu$ is symmetric, the operator is self-adjoint. Note that if $\mu$ is the normalized Haar measure on $G$ then for any nontrivial irrep $\pi_\lambda$ we have $\int_{G}d\mu(g)\pi_\lambda(g)=0$. Following \cite{BANNAI2022108457} we call set $\mathcal{S}$ a $\pi_\lambda$-design if $T_{\nu,\lambda}=0$ and $\delta$-approximate $\pi_\lambda$-design iff $\delta(\lambda):=\|T_{\nu,\lambda}\|\leq\delta$, where $\|\cdot\|$ is the operator norm. We are interested in understanding the spectrum of $T_{\nu,\lambda}$ when $d_\lambda\rightarrow \infty$. To this end we consider the spectral measure $\sigma_{\nu,\lambda}$ of operator $T_{\nu,\lambda}$. Recall that the spectral measure $\sigma_{\nu,\lambda}$ evaluated on an interval $[a,b]\subset\mathbb{R}$ gives:
\begin{equation*}
    \sigma_{\nu,\lambda}([a,b]) := \frac{1}{d_\lambda} (\# \; of \; T_{\nu,\lambda} \; \mathrm{eigenvalues} \;  \mathrm{in} \; [a,b]).
\end{equation*} Its $m$-th moment is
\begin{equation}
\sigma_{\nu,\lambda}^{(m)} =\int x^m d\sigma_{\nu,\lambda}= \frac{1}{d_\lambda}\,\mathrm{Tr}\!\big(T_{\nu,\lambda}^m\big).
\end{equation}
Expanding this trace shows that the moments are weighted sums of characters 
$\chi_\lambda(g)$ normalized by the dimension that is:
\begin{gather}\label{moment}
\sigma_{\nu,\lambda}^{(m)}=\frac{1}{|\mathcal{S}|^m}\sum_{g_1,\ldots , g_m \in \mathcal{S} }\frac{\chi_\lambda\left(g_1\ldots g_m\right)}{d_\lambda}.
\end{gather}
We are interested in the limit of \eqref{moment} when all $d_\lambda^{(i)}$'s are going to infinity. Using Theorem \ref{cor:critdiv} we see that in this limit only words in the generators that reduce to the identity 
contribute. On a free group this happens only through backtracking cancellations, which reproduce the return probabilities of a simple random walk on the infinite $|\mathcal{S}|$-regular tree. Hence the limiting moments coincide with those of the standard Kesten--McKay distribution and the limiting spectral measure is the symmetric 
Kesten--McKay law supported on the interval 
\([-\,\delta_{\mathrm{opt}},\,\delta_{\mathrm{opt}}]\), where
\begin{equation}
\delta_{\mathrm{opt}} = \frac{2\sqrt{s-1}}{s}.
\end{equation}
In the non-symmetric case, one obtains the singular-value version supported 
on \([0,\delta_{\mathrm{opt}}]\).
The largest nontrivial eigenvalue in modulus of $T_{\nu,\lambda}$ converges to 
$\delta_{\mathrm{opt}}$. Thus the spectral gap (defined as $1-\delta_{\mathrm{opt}}$) 
cannot exceed $1-2\sqrt{s-1}/s$ in the large-dimension limit. 
This bound is universal: it depends only on the number of generators $s$, not 
on the specific compact Lie group $G$.

The behavior of the spectral measure is directly related to the problem of 
generating approximate unitary $t$-designs. Random walks on compact groups with distribution $\nu$ converge toward the Haar measure, $\mu$, and the rate of 
convergence is controlled by the spectral gap of the averaging operator. A larger gap means faster mixing and more efficient generation of pseudorandom unitaries.

Our analysis shows that in the large-irrep limit, the normalized characters vanish away from the identity, so the effective spectral statistics reduce to 
the universal Kesten--McKay law. This implies that the spectral radius of the 
nontrivial part of the averaging operator cannot drop below 
$\delta_{\text{opt}}=2\sqrt{s-1}/s$. Equivalently, the spectral gap cannot 
exceed $1-\delta_{\text{opt}}$, no matter which compact semisimple group $G$ we start from.

In practical terms, this means that the ''speed limit'' for producing unitary 
$t$-designs by local random walks is universal: the choice of symmetry group 
cannot yield a better asymptotic spectral gap than the Kesten--McKay bound. 
The only parameter that matters in the limit is the number of generators $s$ 
in the random walk. This connects the algebraic structure of irreps to the 
efficiency of pseudorandomness generation in quantum information tasks.

\section{Conclusions}
\label{sec:conclusion}

In this work, we analyzed the behavior of normalized characters as the dimension of the representation approaches the infinity. We derived a criterion for the vanishing of this limit. The strictly algebraic proof reveals the connection between the structure of the centralizer of the singular element and the asymptotics of the normalized trace. Thanks to this method we were able to precisely delineate  the area of validity of the theorem for the compact, reductive groups.

Our results turn out to have far-reaching implications far beyond the representation theory. Due to their generality, they let us exclude the factor of the kind of the underlying symmetry group of the quantum system in the process of the generation of the pseudorandomness. This gave us a green light to establish generally holding bound that is totally independent of the inner algebraic structure and depends solely on the number of generators we used for construction of our $t$-design. Looking forward, this universal bound provides a strict foundation for exploring the broader dynamics of quantum complexity, such as the extremal jumps in circuit complexity generated by random Hamiltonians \cite{Kotowski2025}.

\section*{Acknowledgments}
The authors thank the anonymous referees for their careful reading of the manuscript and their insightful comments, which contributed to the improvement of this work. This research was funded by the National Science Centre, Poland under the grant OPUS: UMO2020/37/B/ST2/02478.

\bibliographystyle{unsrt}
\bibliography{references}

\end{document}